\documentclass[11pt,oneside]{amsart} 
\usepackage{amsmath, amssymb, verbatim, amscd, amsthm, mathrsfs, mathtools}
\usepackage{color, graphicx, shortvrb}
\usepackage{enumerate}
\usepackage[all]{xy}
\usepackage{mathtools}

\numberwithin{equation}{section}
\theoremstyle{definition} 
\newtheorem{defn}{Definition}[section]

\theoremstyle{plain}
\newtheorem{thm}[defn]{Theorem}

\newtheorem{lem}[defn]{Lemma}
\newtheorem{prop}[defn]{Proposition}

\newtheorem*{rem*}{Remark}

\usepackage[normalem]{ulem}
\usepackage{marginnote}
\usepackage[top=30truemm,bottom=25truemm,left=35truemm,right=35truemm]{geometry}
\newcommand{\q}{\quad} 
 
\newcommand{\qbox}[1]{\q \mbox{#1} \q}

\newcommand{\C}{\mathbb C} 
\newcommand{\N}{\mathbb N} 
\newcommand{\R}{\mathbb R} 
\newcommand{\K}{\mathbb{K}}
\newcommand{\T}{\mathbb{T}}
\newcommand{\F}{\mathcal{F}}
\newcommand{\G}{\mathcal{G}}

\def\Q{\mathbb{Q}} 
\def\lam{\lambda} 

\newcommand{\A}{\mathcal{A}}
\newcommand{\B}{\mathcal{B}}
\newcommand{\supp}{\mathrm{supp}}
\newcommand{\ol}{\overline} 
\newcommand{\set}[1]{\left\{ #1  \right\}}
\newcommand{\norm}[1]{\left\lVert #1 \right\rVert}
\renewcommand{\Re}{\mathrm{Re}}
\renewcommand{\Im}{\mathrm{Im}}

\newcommand{\cx}{C_0(X,\K)}
\newcommand{\cy}{C_0(Y,\K)}

\newcommand{\cxr}{C_0(X,\R)}
\newcommand{\cyr}{C_0(Y,\R)}

\newcommand{\cxc}{C_0(X,\C)}
\newcommand{\cyc}{C_0(Y,\C)}

\newcommand{\coz}{\mathrm{coz}}

\newcommand{\X}{\mathcal{X}}
\newcommand{\Y}{\mathcal{Y}}
\newcommand{\xinf}{x_\infty}
\newcommand{\yinf}{y_\infty}
\newcommand{\HT}{T_0}
\renewcommand{\norm}[1]{\|#1\|}
\def\ap{\alpha_p}

\def\bp{\beta_p}

\author[N.~Shibata]{Natsumi Shibata}
\address[N. Shibata]{Graduate School of Science and Technology,
Niigata University,
Niigata 950-2181, Japan}
\email{f25a056h@mail.cc.niigata-u.ac.jp}

\author[I.~Matsuzaki]{Izuho Matsuzaki}
\address[I.~Matsuzaki]
{Graduate School of Science and Technology,
Niigata University, Niigata 950-2181, Japan}
\email{matsuzaki@m.sc.niigata-u.ac.jp}

\author[T. Miura]{Takeshi Miura}
\address[T. Miura]
{Department of Mathematics,
Faculty of Science,
Niigata University,
Niigata 950-2181, Japan}
\email{miura@math.sc.niigata-u.ac.jp}

\title[Ring isomorphisms in norm]
{Ring isomorphisms in norm between
Banach algebras of continuous functions}

\begin{document}


\begin{abstract}
Let $X$ and $Y$ be locally compact Hausdorff spaces and let
$\K\in\set{\R,\C}$.  We say that a bijection
$T\colon \cx\to\cy$ is a \textit{ring isomorphism in norm} if
\[
\norm{T(f+g)}=\norm{T(f)+T(g)},\qquad
\norm{T(fg)}=\norm{T(f)T(g)}
\]
for every $f,g\in\cx$.  We determine the form of such maps.
When $\K=\C$, under the additional assumption that
$\norm{T(\ol f)}=\norm{T(f)}$ for every $f\in\cxc$, there exist
a continuous function
$w\colon Y\to\set{\lam\in\C:|\lam|=1}$, a homeomorphism
$\varphi\colon Y\to X$, and a closed and open subset
$Y_0\subset Y$ such that
\[
T(f)(y)=
\begin{cases}
w(y)f(\varphi(y)),& y\in Y_0,\\
w(y)\ol{f(\varphi(y))},& y\in Y\setminus Y_0,
\end{cases}
\]
for every $f\in\cxc$ and $y\in Y$.  When $\K=\R$, there exist
a continuous function $w\colon Y\to\set{\pm1}$ and a homeomorphism
$\varphi\colon Y\to X$ such that
\[
T(f)(y)=w(y)f(\varphi(y))
\]
for every $f\in\cxr$ and $y\in Y$.  In particular, the real case
extends the norm version of the Gelfand--Kolmogoroff theorem to the
locally compact setting.
\end{abstract}

\maketitle


\section{Introduction and main results}

In the classical setting where $X$ and $Y$ are compact Hausdorff
spaces, the Banach--Stone theorem \cite{banach,stone} asserts that
every surjective real-linear
isometry from $C(X,\R)$ onto $C(Y,\R)$ is represented
as a weighted composition operator.
On the other hand, the
Gelfand--Kolmogoroff theorem \cite{GelfandKolmogoroff} shows that
every ring isomorphism from $C(X,\R)$ onto $C(Y,\R)$ is induced by a
homeomorphism between the underlying spaces.  More generally, ring
isomorphisms of Banach algebras are also related to classical
automatic continuity questions; see, for example, Kaplansky
\cite{kaplansky}.  Thus maps between spaces of continuous functions
which preserve either the metric structure or the ring structure have
been studied for a long time, and such maps often determine the
topology of the underlying spaces.

Dong, Lin and Zheng \cite{generalizationofGelfandKolmogoroff}
introduced a geometric version of ring isomorphisms on $C(X,\R)$.
In their setting, one does not assume that a bijection $T$ satisfies
\[
T(f+g)=T(f)+T(g)
\quad\text{or}\quad
T(fg)=T(f)T(g).
\]
Instead, the algebraic identities are replaced by the corresponding
norm identities
\[
\norm{T(f+g)}=\norm{T(f)+T(g)},\qquad
\norm{T(fg)}=\norm{T(f)T(g)}
\qquad(f,g\in C(X,\R)).
\]
A bijection satisfying these two conditions is called
a \textit{ring isomorphism in norm}.
Under these assumptions,
$T$ is not assumed to preserve
the metric structure or the ring structure in the usual sense.
Thus the problem is not a direct application of either the
Banach--Stone theorem or the Gelfand--Kolmogoroff theorem; the linear
and algebraic structures have to be recovered from norm identities.
Nevertheless, Dong, Lin and Zheng
\cite{generalizationofGelfandKolmogoroff} proved that such maps on
$C(X,\R)$ are still forced to be weighted composition operators.
The complex-valued case for compact Hausdorff spaces was studied in
\cite{taira}, where ring isomorphisms in norm satisfying the stronger
condition of preserving complex conjugation,
$T(\ol{f})=\ol{T(f)}$, were considered.

The purpose of this paper is to study ring isomorphisms in norm
between spaces $C_0(X,\K)$ and $C_0(Y,\K)$,
where $X$ and $Y$ are locally compact
Hausdorff spaces and $\K=\R$ or $\C$.  There are two additional
difficulties in this setting.  First, if $X$ is not compact, then
$C_0(X,\K)$ does not necessarily have a unit, and hence the arguments
used in the compact case cannot be transferred directly.  In
particular, the boundedness of the map is no longer immediate.
Second, the complex case requires some compatibility with complex
conjugation.  This difficulty is already visible at the scalar level:
when $X$ consists of a single point, the problem involves maps on
$\C$, where pathological behavior may occur without additional
assumptions; see, for example, \cite{charnow,kestelman}.
In the complex case, we impose the following natural norm condition:
\[
\norm{T(\overline f)}=\norm{T(f)}
\qquad(f\in C_0(X,\C)).
\]
This is weaker than the conjugation-preserving condition considered in
\cite{taira}.

We now state the main results of this paper.  Let $X$ and $Y$ be
non-empty locally compact Hausdorff spaces, and let
$\K=\R$ or $\C$.  We denote by $C_0(X,\K)$ the Banach algebra of all
continuous $\K$-valued functions on $X$ vanishing at infinity,
equipped with the supremum norm $\norm{\cdot}$.

A bijection $T\colon C_0(X,\K)\to C_0(Y,\K)$ is called a
\textit{ring isomorphism in norm} if it satisfies
\begin{equation*}
\norm{T(f+g)}=\norm{T(f)+T(g)},\qquad
\norm{T(fg)}=\norm{T(f)T(g)}
\end{equation*}
for every $f,g\in C_0(X,\K)$.

We write
\[
\T=\{\lambda\in\C:|\lambda|=1\}.
\]

\begin{thm}\label{thm:complex}
Let $T\colon \cxc\to\cyc$ be a ring isomorphism in norm.
Assume that
\[
\norm{T(\ol{f})}=\norm{T(f)}
\qquad(f\in C_0(X,\C)).
\]
Then there exist a continuous function
$w\colon Y\to\T$, a homeomorphism
$\varphi\colon Y\to X$, and a closed and open subset
$Y_0\subset Y$ such that
\begin{equation*}
T(f)(y)=
\begin{cases}
w(y)f(\varphi(y)),&y\in Y_0,\\
w(y)\ol{f(\varphi(y))},&y\in Y\setminus Y_0,
\end{cases}
\end{equation*}
for every $f\in\cxc$ and $y\in Y$.

Conversely, every map of this form
is a ring isomorphism in norm
and satisfies the above conjugation condition.
\end{thm}

\begin{thm}\label{thm:real}
Let $T\colon \cxr\to\cyr$ be a ring isomorphism in norm.
Then there exist a continuous function
$w\colon Y\to\{\pm1\}$ and a homeomorphism
$\varphi\colon Y\to X$ such that
\begin{equation*}
T(f)(y)=w(y)f(\varphi(y))
\end{equation*}
for every $f\in\cxr$ and $y\in Y$.

Conversely, every map of this form
is a ring isomorphism in norm.
\end{thm}

In particular, Theorem~\ref{thm:real} extends the norm version of the
Gelfand--Kolmogoroff theorem of Dong, Lin and Zheng
\cite{generalizationofGelfandKolmogoroff} to the locally compact
setting.

We next describe the main idea of the proof.  The crucial point is to
prove the automatic boundedness of the map obtained from a ring
isomorphism in norm.  We pass to the one-point compactifications
$\X$ and $\Y$ of $X$ and $Y$, respectively, and identify
$C_0(X,\K)$ and $C_0(Y,\K)$ with the closed ideals
\[
A=\{f\in C(\X,\K): f(\xinf)=0\},
\qquad
B=\{u\in C(\Y,\K): u(\yinf)=0\}.
\]
The original map $T$ then induces a map $\HT\colon A\to B$.
We first show that $\HT$ and $\HT^{-1}$ are additive and separating.
Separating maps and related zero-product preserving maps have been
studied in connection with Banach--Stone type theorems and automatic
continuity; see, for example,
\cite{leungtang,fonthernandez,jarosz,keliwong}.
Most results in this direction concern linear maps.  In contrast, a
ring isomorphism in norm is not assumed to be linear a priori.

Using the separating property and the supports of functions in $A$,
we construct a support map from $Y$ into $\X$.  This map may take
values at the point at infinity, and this is one of the main
difficulties caused by the absence of a unit in $C_0(X,\K)$.
We then prove the boundedness of the
additive evaluation maps
\[
f\longmapsto \HT(f)(y)
\qquad(y\in Y).
\]
Once this boundedness is obtained, the Banach--Steinhaus theorem gives
the boundedness of $\HT$, and additivity implies real-linearity.  A
standard power argument using the norm identities shows that $\HT$ is
a surjective real-linear isometry.  The desired representations then
follow from the classical Banach--Stone theorem in the real case and,
in the complex case,
from a known Banach--Stone type theorem for uniformly closed
function algebras, applied here to $C_0(X,\C)$ and $C_0(Y,\C)$
\cite{miurareallinear}.

The rest of this paper is organized as follows.
In Section~\ref{sect:additive}, we show that the induced map and its
inverse are additive and separating.  We then use the supports of
functions to construct a map from $Y$ into the one-point
compactification of $X$.
In Section~\ref{sect:bounded}, we investigate the boundedness of
evaluation functionals and prove the boundedness of the induced map
$\HT$ under the assumptions of the main theorems.
In Section~\ref{sect:proof}, we complete the proofs of
Theorems~\ref{thm:complex} and~\ref{thm:real} by combining this
boundedness with known descriptions of surjective real-linear
isometries.

\section{Additive and separating maps}
\label{sect:additive}

In this section, we associate with the original map $T$ a map
$\HT$ between certain ideals of the algebras of continuous functions
on the one-point compactifications of $X$ and $Y$.
We show that $\HT$ and its inverse are additive and separating.
These properties will be used to construct, by means of supports,
a map from $Y$ into the one-point compactification of $X$.

Let $\A$ and $\B$ be Banach algebras over $\K$.
A bijection $S\colon\A\to\B$ is called a
\textit{ring isomorphism in norm} if
\[
\norm{S(a+b)}=\norm{S(a)+S(b)},\qquad
\norm{S(ab)}=\norm{S(a)S(b)}
\]
for every $a,b\in\A$.
A map $S\colon\A\to\B$ is said to be
\textit{separating} if
\[
ab=0
\quad\Longrightarrow\quad
S(a)S(b)=0
\]
for every $a,b\in\A$.

Since $X$ and $Y$ need not be compact,
it is convenient to work
on their one-point compactifications.
This allows us to use compactness
arguments while keeping track of the fact
that functions in $C_0(X,\K)$ vanish at infinity.

Let $\X$ and $\Y$ be the one-point compactifications
of $X$ and $Y$, respectively.
Let $\xinf\in\X$ and $\yinf\in\Y$ denote
the points at infinity.
When $X$ is compact,
we regard its one-point compactification
as the disjoint union $X\cup\{\xinf\}$,
where $\xinf$ is isolated.
The same convention is applied to $Y$.

Unless otherwise specified,
if $E$ is a subset of $\X$ or $\Y$,
then $\ol{E}$ denotes its closure
in the respective ambient compact
space.
When $\K=\C$, the symbol $\overline{f}$ denotes the pointwise
complex conjugate of a complex-valued function $f$.

We define $A\subset C(\X,\K)$
and $B\subset C(\Y,\K)$:
\[
A=\{f\in C(\X,\K):f(\xinf)=0\},
\qquad
B=\{u\in C(\Y,\K):u(\yinf)=0\}.
\]
Then $A$ and $B$ are both closed ideals
of $C(\X,\K)$ and $C(\Y,\K)$, respectively.
For each $f\in A$, we have $f|_X\in C_0(X,\K)$.
Conversely, for each $g\in C_0(X,\K)$,
we define
\[
\hat{g}(x)=
\begin{cases}
g(x),&x\in X,\\
0,&x=\xinf.
\end{cases}
\]
Since $g$ vanishes at infinity, $\hat{g}$ is continuous on $\X$.
Thus $\hat{g}\in A$ for all $g\in C_0(X,\K)$.
Each function in $C_0(X,\K)$ has a unique continuous extension
to $\X$ which vanishes at $\xinf$, and this identifies
$C_0(X,\K)$ isometrically with $A$.
The same observation applies to $C_0(Y,\K)$ and $B$.

Let $T\colon C_0(X,\K)\to C_0(Y,\K)$ be
a ring isomorphism in norm.
Thus $T$ satisfies the two norm identities above with
$\A=C_0(X,\K)$ and $\B=C_0(Y,\K)$.
Via the above identifications, $T$ induces
a map $\HT\colon A\to B$.
Explicitly, it is given by
\begin{equation}\label{HT}
\HT(f)(y)=
\begin{cases}
T(f|_X)(y),&y\in Y,\\
0,&y=\yinf,
\end{cases}
\qquad(f\in A).
\end{equation}
It follows directly from the definition of $\HT$ and from the norm
identities for $T$ that $\HT\colon A\to B$ satisfies
the same norm identities:
\[
\|\HT(f+g)\|
=\|\HT(f)+\HT(g)\|,\qquad
\|\HT(fg)\|
=\|\HT(f)\HT(g)\|
\qquad(f,g\in A).
\]

When $\K=\C$, we assume that
$\|T(\ol{f})\|=\|T(f)\|$
for every $f\in C_0(X,\C)$.
For each $f\in A$, we have $f|_X\in C_0(X,\C)$ and
\[
(\ol f)|_X=\overline{f|_X}.
\]
Hence, by the definition of $\HT$ and the conjugation condition for
$T$,
\[
\|\HT(\ol f)\|
=\|T(\overline{f|_X})\|
=\|T(f|_X)\|
=\|\HT(f)\|.
\]

The map $\HT$ is bijective.
Indeed, let $f,g\in A$ satisfy
$\HT(f)=\HT(g)$.
Then $T(f|_X)=T(g|_X)$ on $Y$
by definition.
Since $T$ is injective, we have
$f=g$ on $X$.
In addition, $f(\xinf)=0=g(\xinf)$,
since $f,g\in A$.
Thus $f=g$ on $\X$.
This shows that $\HT$ is injective.
Now we fix $u\in B$ to prove that
$\HT$ is surjective.
Since $u|_Y\in C_0(Y,\K)$,
the surjectivity of $T$ shows that
there exists $h\in C_0(X,\K)$
such that $T(h)=u|_Y$.
By the definition of $\HT$ and $B$,
we have $\HT(\hat{h})(\yinf)=0=u(\yinf)$.
Hence $\HT(\hat{h})=u$ on $\Y$.
This shows that $\HT$ is surjective.
Thus $\HT\colon A\to B$ is a bijection
satisfying the same norm identities as $T$.
In the complex case, it also satisfies
\[
\norm{\HT(\ol{f})}=\norm{\HT(f)}
\qquad(f\in A).
\]
In what follows, we regard $\HT$ as the induced ring isomorphism
in norm from $A$ onto $B$.
The original map $T$ will be recovered from $\HT$ through
the above identifications.

For $f\in A$, we set
\[
\coz(f)=\{x\in\X:f(x)\ne0\},
\]
and denote by $\supp(f)$
the closure of $\coz(f)$ in $\X$.
Since $f(\xinf)=0$, we have $\coz(f)\subset X$;
however, $\supp(f)$ may contain $\xinf$.
We use the same notation for functions
in $B$, where the closure is taken in $\Y$.

In what follows, we work with this induced map
$\HT\colon A\to B$.
Thus $\HT$ is a ring isomorphism in norm.
In the complex case, it also satisfies
\[
\norm{\HT(\ol{f})}=\norm{\HT(f)}
\qquad(f\in A).
\]

The first step is to recover algebraic information
from the two norm identities.
The following proposition shows that,
although $\HT$ is not assumed to be additive
or multiplicative, it is automatically
additive and separating.

\begin{prop}\label{AdditiveSeparating}
The maps $\HT$ and $\HT^{-1}$ are additive and separating.
\end{prop}

\begin{proof}
By the additive norm identity for $\HT$, we have 
$\norm{\HT(f+g)}
=\norm{\HT(f)+\HT(g)}$
for all $f,g\in A$.
Hence, by \cite[Corollary~1]{Tabor}, $\HT$ is additive.
In particular, $\HT(0)=0$.

We prove that $\HT$ is a separating map.
That is,
\[
    fg= 0
    \quad\Longrightarrow\quad
    \HT(f)\HT(g)= 0
\]
for all $f,g\in A$.
Let $f,g\in A$.
If $fg=0$, then,
by the product norm identity for $\HT$,
\[
\norm{\HT(f)\HT(g)}
    =\norm{\HT(fg)}
    =\norm{\HT(0)}
    =0.
    \]
Hence 
    $\HT(f)\HT(g)=0$.
Thus $\HT$ is a separating map.

Since $\HT$ is additive and bijective,
its inverse $\HT^{-1}$ is also additive.
We prove that $\HT^{-1}$ is separating.
Let $u,v\in B$ with $uv=0$.
Since $\HT$ is surjective, there exist
$f_0,g_0\in A$ such that
$\HT(f_0)=u$ and $\HT(g_0)=v$.
Then,
by the product norm identity for $\HT$,
we have
\[
\norm{\HT(f_0g_0)}
=\norm{\HT(f_0)\HT(g_0)}
=\norm{uv}=0.
\]
Thus $\HT(f_0g_0)=0$.
Since $\HT$ is injective, we obtain
$f_0g_0=0$.
Therefore $\HT^{-1}(u)\HT^{-1}(v)=0$.
Hence $\HT^{-1}$ is also
a separating map.
\end{proof}

By Proposition~\ref{AdditiveSeparating},
$\HT$ is additive.  Hence $\HT$ is $\Q$-linear, that is,
\[
\HT(qf)=q\HT(f)
\qquad(q\in\Q,\ f\in A).
\]

We next use the separating property to localize functions at points
of $Y$.  For each $y\in Y$, we consider the family of functions whose
images under $\HT$ do not vanish at $y$.
For each $y\in Y$, define 
\[
\F_y
=
\{f\in A:\HT(f)(y)\neq0\}.
\]
Let $y\in Y$.
By Urysohn's lemma, there exists
$u\in C(\Y,\R)$ such that
$u(y)=1$ and $u(\yinf)=0$.
Thus $u\in B$.
Since $\HT$ is surjective,
there exists
$f\in A$ such that $\HT(f)=u$.
Hence $\HT(f)(y)=u(y)=1$.
Thus $\F_y\neq\emptyset$
for every $y\in Y$.

The next lemma is a key step
to show that the intersection of
supports of functions in $\F_y$
is non-empty.

\begin{lem}\label{lem:finiteintersection}
Let $f_1,\dots,f_n$ be finitely many
elements of $A$
such that $\bigcap_{j=1}^n\supp(f_j)=\emptyset$.
For each $y\in Y$ there exists $k\in\{1,\dots,n\}$
such that $f_k\notin\F_y$.
\end{lem}

\begin{proof}
Set $U_j=\X\setminus\supp(f_j)$
for $j=1,\dots,n$.
Since $\bigcap_{j=1}^n\supp(f_j)=\emptyset$,
we have $\bigcup_{j=1}^nU_j=\X$.
There exists a partition of unity
$g_1,\dots,g_n\in C(\X,\R)$
subordinate to this cover, such that
\[
\sum_{j=1}^ng_j=1,\qquad
\supp(g_j)\subset U_j
\qquad(j=1,\dots,n).
\]
Thus
$f_1=f_1\sum_{j=1}^ng_j
=\sum_{j=1}^nf_1g_j$.
Since $g_j\in C(\X,\K)$ and
$A$ is an ideal of $C(\X,\K)$,
we have $f_1g_j\in A$
for $j=1,\dots,n$.
The additivity of $\HT$ yields
\[
\HT(f_1)
=\HT\left(\sum_{j=1}^nf_1g_j\right)
=\sum_{j=1}^n\HT(f_1g_j).
\]

Now we fix an arbitrary $y\in Y$.
If $\HT(f_1g_j)(y)=0$ for
$j=1,\dots,n$,
then $\HT(f_1)(y)=0$.
This shows that $f_1\notin\F_y$.

It remains to consider the case where
$\HT(f_1g_k)(y)\neq0$ for some
$k\in\{1,\dots,n\}$.
For such $k$, we have
\[
\supp(f_1g_k)
\subset\supp(g_k)
\subset U_k
=\X\setminus\supp(f_k).
\]
This shows that
\[
\supp(f_1g_k)\cap\supp(f_k)
=\emptyset.
\]
Thus $(f_1g_k)f_k=0$.
Since $\HT$ is separating
by Proposition~\ref{AdditiveSeparating},
we have
$\HT(f_1g_k)\HT(f_k)=0$.
Hence
\[
\HT(f_1g_k)(y)\HT(f_k)(y)=0.
\]
By assumption,
$\HT(f_1g_k)(y)\neq0$,
which shows $\HT(f_k)(y)=0$.
Thus $f_k\notin\F_y$.
\end{proof}

The following lemmas show that
this family determines a unique point
of the one-point compactification of $X$.
This point will be denoted
by $\sigma(y)$.

\begin{lem}\label{nonempty}
For every $y\in Y$, the intersection
$\bigcap_{f\in\F_y}\supp(f)$ is non-empty.    
\end{lem}

\begin{proof}
Since $\X$ is compact, it is sufficient to prove
that the family $\{\supp(f):f\in\F_y\}$ has
the finite intersection property.
Let $f_1,\dots,f_n$ be finitely many
elements of $\F_y$.
Applying Lemma~\ref{lem:finiteintersection}
contrapositively, we obtain
$\bigcap_{j=1}^n\supp(f_j)\neq\emptyset$.
Therefore the family $\{\supp(f):f\in\F_y\}$ has
the finite intersection property.
\end{proof}

\begin{lem}\label{unique}
    For every $y\in Y$, the intersection
$\bigcap_{f\in \F_y}\supp(f)$ is a singleton. 
\end{lem}

\begin{proof}
By Lemma~\ref{nonempty},
there exists $x_0\in\bigcap_{f\in\F_y}\supp(f)$.
Let $U$ be an arbitrary open neighborhood
of $x_0$ in $\X$.
Since $\X$ is compact Hausdorff,
there exists an open neighborhood $V$ of $x_0$
with $x_0\in V\subset \ol{V}\subset U$.
By Urysohn's lemma,
there exists $g\in C(\X,\R)$ such that 
\begin{equation*}
\supp(g)\subset U,\qquad
g=1\qbox{on}\ol{V}.
\end{equation*}
Since $\F_y\neq\emptyset$,
choose $f_0\in\F_y$.
Then $f_0\in A$ and
\[
\HT(f_0)(y)\neq0.
\]

Since $g\in C(\X,\K)$ and
$A$ is an ideal of $C(\X,\K)$,
we have $f_0g\in A$.
Since $g=1$ on $V$, we obtain
$f_0-f_0g=0$ on $V$.
Thus $x_0\notin\supp(f_0-f_0g)$.
By the choice of $x_0$, this implies
that $f_0-f_0g\notin\F_y$.
Hence $\HT(f_0-f_0g)(y)=0$.
The additivity of $\HT$ gives
$\HT(f_0)(y)=\HT(f_0g)(y)$.
Since $f_0\in\F_y$, we obtain
$\HT(f_0g)(y)=\HT(f_0)(y)\neq0$,
which implies $f_0g\in\F_y$.
Therefore
\[
x_0
\in\bigcap_{f\in\F_y}\supp(f)
\subset\supp(f_0g)
\subset\supp(g)
\subset U.
\]
Since $U$ was arbitrary
and $\X$ is Hausdorff, we obtain
$\bigcap_{f\in\F_y}\supp(f)=\{x_0\}$.
\end{proof}

By Lemma~\ref{unique},
for each $y\in Y$,
the intersection
$\bigcap_{f\in\F_y}\supp(f)$
consists of a single point.
Hence we define a map
$\sigma:Y\to\X$
by
\[
\{\sigma(y)\}
=
\bigcap_{f\in\F_y}\supp(f)
\qquad (y\in Y).
\]
By the definition of $\sigma$,
if $f\in\F_y$ then $\sigma(y)\in\supp(f)$.
That is,
\begin{equation}\label{supp}
\HT(f)(y)\neq0
\quad\Longrightarrow\quad
\sigma(y)\in\supp(f)
\end{equation}
for every $f\in A$ and $y\in Y$.
Equivalently,
\[
\sigma(y)\notin\supp(f)
\quad\Longrightarrow\quad
\HT(f)(y)=0
\]
for every $f\in A$ and $y\in Y$.

The map $\sigma$ records the point
at which functions in $A$ must have support
whenever their images do not vanish at $y$.
The next lemma shows that
this assignment is continuous.

\begin{lem}\label{conti}
    The map $\sigma\colon Y\to \X$ is continuous. 
\end{lem}

\begin{proof}
    Let $U$ be an open subset of $\X$. 
    We show that $\sigma^{-1}(U)$ is an open subset of $Y$. 
If $U=\X$, then $\sigma^{-1}(U)=Y$, and there is nothing to prove.
Thus we may assume $U\neq\X$.
    Fix $y_0\in \sigma^{-1}(U)$.
    The definition of $\sigma(y_0)$ shows that
$\bigcap_{f\in \F_{y_0}}\supp(f)
=\set{\sigma(y_0)}\subset U$,
which implies that
\[
\X\setminus U
\subset\bigcup_{f\in\F_{y_0}}(\X\setminus\supp(f)).
\]
   Since  $\X\setminus U$ is compact,
    there exist
    $f_1,\dots,f_n\in\F_{y_0}$
    such that
  \[
  \X\setminus U
    \subset
    \bigcup_{j=1}^n
    (\X\setminus \supp(f_j)).
    \]
    Hence we have
    \[
    \bigcap_{j=1}^n\supp(f_j)\subset U.
    \]
Set $O=\bigcap_{j=1}^n\coz(\HT(f_j))$. 
Since $f_j\in\F_{y_0}$, we have
$\HT(f_j)(y_0)\neq0$
for $j=1,\dots,n$.
Thus, $O$ is an open neighborhood of $y_0$.

We prove $O\subset\sigma^{-1}(U)$.
Fix $y\in O$.
Then $\HT(f_j)(y)\neq0$ for $j=1,\dots,n$
by definition.
It follows from \eqref{supp} that
$\sigma(y)\in\supp(f_j)$ for $j=1,\dots,n$. 
Hence
\[
\sigma(y)\in \bigcap_{j=1}^n\supp(f_j)\subset U.
\]
Thus $y_0\in O\subset\sigma^{-1}(U)$.
Since $y_0\in\sigma^{-1}(U)$ was chosen arbitrarily,
$\sigma^{-1}(U)$ is open.
Therefore $\sigma$ is continuous.
\end{proof}

Recall that the inverse map $\HT^{-1}$
of $\HT$ is also additive and separating
by Proposition~\ref{AdditiveSeparating}. 
Hence, by applying similar arguments to
$\HT^{-1}$, we obtain a map
$\tau\colon X\to\Y$ defined by
$\set{\tau(x)}=\bigcap_{u\in\G_x}\supp(u)$
for every $x\in X$, where 
\[
\G_x=\{u\in B:\HT^{-1}(u)(x)\neq 0\}.
\] 
Moreover, the same argument as in
Lemma~\ref{conti} shows that
$\tau$ is continuous.
For each $u\in B$ and $x\in X$,
we have
\begin{equation}\label{eq:tauproperty}
\HT^{-1}(u)(x)\neq0
\quad\Longrightarrow\quad
\tau(x)\in\supp(u).
\end{equation}

\begin{lem}\label{lem:denseopen}
The set $\sigma^{-1}(X)$ is a dense open subset of $Y$.
\end{lem}

\begin{proof}
By Lemma~\ref{conti},
$\sigma\colon Y\to\X$ is continuous.
Thus $\sigma^{-1}(\{\xinf\})$ is a closed subset of $Y$.
Hence
$\sigma^{-1}(X)
=Y\setminus\sigma^{-1}(\{\xinf\})$
is open in $Y$.

We prove that $\sigma^{-1}(X)$ is dense in $Y$.
Let $y\in Y$, and let $O$ be an open neighborhood of $y$ in $Y$.
Since $Y$ is open in $\Y$, the set $O$ is also open in $\Y$.
By Urysohn's lemma, there exists $u\in C(\Y,\R)$ such that
\[
u(y)=1,\qquad
\supp(u)\subset O.
\]
Then $u\in B$.

Since $u\neq0$ and $\HT$ is bijective,
we have $\HT^{-1}(u)\neq0$.
As $\HT^{-1}(u)\in A$, there exists $x_0\in X$ such that
\[
\HT^{-1}(u)(x_0)\neq0.
\]
Since $X$ is open in $\X$, by Urysohn's lemma there exists
$f_0\in C(\X,\R)$ such that
\[
f_0(x_0)=1,\qquad
\supp(f_0)\subset X.
\]
Then $f_0\in A$, and
$f_0\HT^{-1}(u)\neq0$.
Since $\HT^{-1}$ is separating, this implies
$\HT(f_0)u\neq0$.
Hence there exists $y_0\in Y$ such that
\[
\HT(f_0)(y_0)\neq0,
\qquad
u(y_0)\neq0.
\]
By \eqref{supp}, we have
$\sigma(y_0)\in\supp(f_0)\subset X$.
Thus $y_0\in\sigma^{-1}(X)$.
Moreover, since $u(y_0)\neq0$ and $\supp(u)\subset O$,
we have $y_0\in O$.
Therefore
$\sigma^{-1}(X)\cap O\neq\emptyset$.
Since $y\in Y$ and the open neighborhood $O$ of $y$ were arbitrary,
$\sigma^{-1}(X)$ is dense in $Y$.
\end{proof}

We next compare the two maps $\sigma$ and $\tau$.
This comparison shows that
$\sigma$ is injective on $\sigma^{-1}(X)$.

\begin{lem}\label{bijection}
The restriction
$\sigma|_{\sigma^{-1}(X)}\colon \sigma^{-1}(X)\to X$
is injective.
\end{lem}

\begin{proof}
Let $y\in \sigma^{-1}(X)$, and put $x=\sigma(y)$.
Then $x\in X$.
We first prove that
\[
y\in\bigcap_{u\in\G_x}\supp(u).
\]
Let $u\in\G_x$.
We show that $y\in\supp(u)$.
Let $O$ be an arbitrary open neighborhood of $y$ in $\Y$.
Since $y\in Y$ and $Y$ is open in $\Y$,
we may assume that $O\subset Y$.
By Urysohn's lemma, there exists $v\in C(\Y,\R)$ such that
\[
v(y)=1,\qquad
\supp(v)\subset O\subset Y.
\]
Then $v\in B$.
Since
$\HT(\HT^{-1}(v))(y)=v(y)=1$,
\eqref{supp} yields
$\sigma(y)\in\supp(\HT^{-1}(v))$.
Thus
\[
x\in\supp(\HT^{-1}(v)).
\]
On the other hand, since $u\in\G_x$, we have
$\HT^{-1}(u)(x)\neq0$.
Hence $\coz(\HT^{-1}(u))$ is an open neighborhood of $x$.
Since $x\in\supp(\HT^{-1}(v))$, it follows that
\[
\coz(\HT^{-1}(u))\cap\coz(\HT^{-1}(v))\neq\emptyset.
\]
Therefore
$\HT^{-1}(u)\HT^{-1}(v)\neq0$.
Since $\HT^{-1}$ is separating, we have $uv\neq0$.
Hence
\[
\coz(u)\cap\coz(v)\neq\emptyset.
\]
Since $\supp(v)\subset O$, we obtain
\[
\emptyset\neq\coz(u)\cap\coz(v)
\subset \coz(u)\cap O.
\]
Thus $\coz(u)\cap O\neq\emptyset$.
Since $O$ was arbitrary, we conclude that
$y\in\supp(u)$.
As $u\in\G_x$ was arbitrary,
\[
y\in\bigcap_{u\in\G_x}\supp(u)=\{\tau(x)\}.
\]
Therefore
$y=\tau(x)=\tau(\sigma(y))$.

We now prove that $\sigma$ is injective on $\sigma^{-1}(X)$.
Let $y_1,y_2\in\sigma^{-1}(X)$ and suppose that
$\sigma(y_1)=\sigma(y_2)$.
By the identity just proved,
\[
y_1
=\tau(\sigma(y_1))
=\tau(\sigma(y_2))
=y_2.
\]
Hence the restriction
$\sigma|_{\sigma^{-1}(X)}\colon \sigma^{-1}(X)\to X$
is injective.
\end{proof}

These results provide the support-theoretic
framework needed in the next section,
where we prove the boundedness of the 
additive evaluation maps
associated with $\HT$.

\section[The boundedness of the induced map]{The boundedness of $\HT$}
\label{sect:bounded}

In this section, we prove the boundedness of the induced map $\HT$.
For each $y\in Y$, define
\begin{equation}\label{phiy}
\phi_y(f)=\HT(f)(y)\qquad(f\in A).
\end{equation}
Since $\HT$ is additive, so is $\phi_y$.

In what follows, an additive map $\phi\colon A\to\K$ is said to be
bounded if there exists a constant $C>0$ such that
\[
|\phi(f)|\le C\norm{f}
\qquad(f\in A).
\]
Similarly, the additive map $\HT\colon A\to B$
is said to be bounded if
there exists a constant $M>0$
such that
\[
\norm{\HT(f)}\leq M\norm{f}
\qquad(f\in A).
\]
The main point of this section is to prove that each evaluation map
$\phi_y$ is bounded. Once this is known, these additive maps are
real-linear, and the boundedness of $\HT$ will follow from the
Banach--Steinhaus theorem.

\subsection*{Isolated points}

We first consider evaluation maps at isolated points of $Y$.
This discussion is relevant only when $Y$ has isolated points.
At such points, the support construction gives a much more rigid
description of the corresponding evaluation map.

In the rest of this subsection,
let $p$ be an isolated point of $Y$.
We denote by $e_p\colon\Y\to\K$
the function taking the value $1$ at $p$
and $0$ elsewhere.
Since $p$ is an isolated point of $Y$,
we have $e_p\in B$.

The following lemma shows that an isolated point of $Y$ corresponds,
through the support construction, to an isolated point of $\X$.

\begin{lem}\label{isolated}
$\sigma(p)$ is
an isolated point of $\X$.
Moreover, we have $\sigma(p)\in X$.
\end{lem}

\begin{proof}
Set $f=\HT^{-1}(e_p)$.
We prove
$\coz(f)\subset\{\sigma(p)\}$.
Let $x\in\coz(f)$ and 
let $U$ be an arbitrary open neighborhood
of $x$ in $\X$.
Since $X$ is open in $\X$,
we may assume that $U\subset X$.
    By Urysohn's lemma, there exists
$g\in C(\X,\R)$ such that 
\begin{equation*}
g(x)=1,\qquad
\supp(g)\subset U.
\end{equation*}
Since $x\in\coz(f)$, we have
$fg\neq0$.
As $\supp(g)\subset U\subset X$,
we have $g\in A$.
Since $\HT^{-1}$ is separating
and $fg\neq0$,
we have $\HT(f)\HT(g)\neq0$.
In particular, $\HT(g)(p)\neq0$.
Applying \eqref{supp},
we obtain $\sigma(p)\in\supp(g)\subset U$.
Since $U$ was arbitrary and $\X$ is Hausdorff,
we conclude that $\sigma(p)=x$.
Therefore $\coz(f)\subset\{\sigma(p)\}$.

Since $e_p\neq 0$
and $\HT(0)=0$, we have 
$f=\HT^{-1}(e_p)\neq 0$, and thus $\coz(f)\neq \emptyset$.  
Since $\emptyset\neq\coz(f)\subset\{\sigma(p)\}$,
we obtain
$\{\sigma(p)\}=\coz(f)$.
Since $\coz(f)$ is open in $\X$,
the singleton $\{\sigma(p)\}$ is open in $\X$.
We conclude that $\sigma(p)$ is an isolated point of $\X$.

Since $e_p\in B$, we have
$f=\HT^{-1}(e_p)\in A$.
Thus $\{\sigma(p)\}=\coz(f)\subset X$,
and consequently $\sigma(p)\in X$.
\end{proof}

By Lemma~\ref{isolated},
$\sigma(p)\in X$ is an isolated point of $\X$.
Hence the function $e_{\sigma(p)}$ is continuous on $\X$
and vanishes at $\xinf$.
Thus $e_{\sigma(p)}\in A\setminus\{0\}$.

We record the values of $\HT$ at $p$ for functions of the form
$\lambda e_{\sigma(p)}$.
Define a map $\ap\colon \K\to\K$ by
\[
\ap(\lam)=\HT(\lam e_{\sigma(p)})(p)
\qquad(\lam\in\K).
\]
Since $\HT$ is additive, so is $\ap$.

The following lemma shows that the image of each scalar multiple
of $e_{\sigma(p)}$ is supported at the single point $p$, and records
the norm relations satisfied by $\ap$.

\begin{lem}\label{isolatedform}
    The following identities hold for every $\lam\in\K$:
\[
\HT(\lam e_{\sigma(p)})=\ap(\lam)e_p,\qquad
\norm{\HT(\lam e_{\sigma(p)})}=|\ap(\lam)|,\qquad
|\ap(\lam^2)|=|\ap(\lam)|^2.
\]
In particular, $|\ap(1)|=1$.

If, in addition, $\K=\C$, then
we also have
\[
|\ap(i)|=1,\qquad
|\ap(\ol{\lam})|=|\ap(\lam)|
\qquad(\lam\in\C).
\]
\end{lem}

\begin{proof}
    We first show that
$\HT(\lam e_{\sigma(p)})=\ap(\lam)e_p$
for every $\lam\in\K$. 
    If $\lam=0$, the identity follows immediately
from $\HT(0)=0$.
    Assume that $\lam\neq0$.
    The definition of $\ap$ shows that 
    \[
\HT(\lam e_{\sigma(p)})(p)
=\ap(\lam)=\ap(\lam)e_p(p).
    \]
We prove that $\HT(\lam e_{\sigma(p)})=0$
on $Y\setminus\{p\}$.
Suppose that $y\in Y$ satisfies
$\HT(\lam e_{\sigma(p)})(y)\neq0$. 
    It follows from \eqref{supp} that
$\sigma(y)\in\supp(\lam e_{\sigma(p)})$. 
    Since $\lam\neq0$, we have
$\coz(\lam e_{\sigma(p)})
=\coz(e_{\sigma(p)})=\set{\sigma(p)}$.
Hence
\[
\supp(\lam e_{\sigma(p)})
=\set{\sigma(p)}.
\]
Therefore $\sigma(y)=\sigma(p)$.
By Lemma~\ref{isolated},
$\sigma(p)\in X$;
hence $y,p\in\sigma^{-1}(X)$.
Since $\sigma$ is injective
on $\sigma^{-1}(X)$ by Lemma~\ref{bijection},
we obtain $y=p$.
This shows that
$\HT(\lam e_{\sigma(p)})=0$ on $Y\setminus\{p\}$.
Since $\HT(\lam e_{\sigma(p)})\in B$, it also vanishes at $\yinf$.
Hence
\[
\HT(\lam e_{\sigma(p)})=0
\quad\text{on } \Y\setminus\{p\}.
\]
Since $e_p=0$ on $\Y\setminus\{p\}$,
we obtain
\[
\HT(\lam e_{\sigma(p)})=\ap(\lam)e_p
\qquad\mbox{on $\Y\setminus\{p\}$}.
\]
Since $\HT(\lam e_{\sigma(p)})(p)=\ap(\lam)e_p(p)$,
we conclude that $\HT(\lam e_{\sigma(p)})=\ap(\lam)e_p$.
Taking norms, we obtain
$\norm{\HT(\lam e_{\sigma(p)})}
=|\ap(\lam)|$ for every $\lambda\in\K$.

Since $(e_{\sigma(p)})^2=e_{\sigma(p)}$, the product norm identity
for $\HT$ gives
\[
    |\ap(\lam^2)|
=\norm{\HT(\lam^2e_{\sigma(p)})}
=\norm{\HT\big((\lam e_{\sigma(p)})^2\big)}
=\norm{\HT(\lam e_{\sigma(p)})}^2
=|\ap(\lam)|^2.
\]

Finally we prove $|\ap(1)|=1$.
Substituting \(\lambda=1\) into the identity
\(|\ap(\lambda^2)|=|\ap(\lambda)|^2\) yields
\[
    |\ap(1)|=|\ap(1)|^2.
\]
Thus $|\ap(1)|\in\{0,1\}$.
By the identity already proved,
$\HT(e_{\sigma(p)})=\ap(1)e_p$.
Since $e_{\sigma(p)}\neq0$
and $\HT$ is injective,
we have $\HT(e_{\sigma(p)})\neq\HT(0)=0$.
Thus $\ap(1)e_p\neq0$,
which yields $\ap(1)\neq0$.
Since $|\ap(1)|\in\{0,1\}$,
we conclude $|\ap(1)|=1$.

Now we consider the case where
$\K=\C$.
Substituting $\lam=i$ into
the preceding equality yields
$|\ap(-1)|=|\ap(i)|^2$.
By the additivity of $\ap$,
we have $\ap(-1)=-\ap(1)$.  
Hence $|\ap(i)|^2=|\ap(-1)|=|\ap(1)|=1$,
and thus $|\ap(i)|=1$.

Recall that
$\norm{\HT(\ol{f})}=\norm{\HT(f)}$
for every $f\in A$.
By the identity obtained above,
\[
|\ap(\ol{\lam})|
=\norm{\HT(\ol{\lam}e_{\sigma(p)})}
=\norm{\HT(\ol{\lam e_{\sigma(p)}})}
=\norm{\HT(\lam e_{\sigma(p)})}
=|\ap(\lam)|
\]
for every $\lam\in \C$. 
\end{proof}

The next lemma shows that the value of $\HT(f)$ at $p$ depends only
on the value of $f$ at the single point $\sigma(p)$.

\begin{lem}\label{formoff}
    The identity $\HT(f)(p)=\ap\big(f(\sigma(p))\big)$
holds for every $f\in A$.
\end{lem}

\begin{proof}
    Fix $f\in A$.
    Set $\lam=f(\sigma(p))$ and
$g=f-\lam e_{\sigma(p)}$.
    Then $g(\sigma(p))=0$. 
    Since $\sigma(p)$ is an isolated point
of $\X$ by Lemma~\ref{isolated}, 
    $\sigma(p)\notin \supp(g)$.
Hence the contrapositive of
\eqref{supp} yields $\HT(g)(p)=0$.
    The additivity of $\HT$ and
Lemma~\ref{isolatedform} show that
\[
\HT(f)(p)
=\HT(g)(p)+\HT(\lam e_{\sigma(p)})(p)
=\ap(\lam)e_p(p)
=\ap\big(f(\sigma(p))\big).
\]
This completes the proof.
\end{proof}

We normalize this scalar map in order to reduce the problem to an
additive map on $\K$ which fixes $1$.
Define a map $\bp\colon \K\to\K$ by 
\[
\bp(\lam)=\frac{\ap(\lam)}{\ap(1)}
\qquad (\lam \in \K).
\]
Then $\bp$ is well-defined,
since $\ap(1)\neq0$
by Lemma~\ref{isolatedform}.

The next lemma records the algebraic and norm-preserving properties
of $\bp$ inherited from $\ap$.

\begin{lem}\label{lem:betait}
The map $\bp\colon\K\to\K$ is an additive
map with $\bp(1)=1$ and
\begin{equation}\label{betasquared}
|\bp(\lam^2)|=|\bp(\lam)|^2
\qquad(\lambda\in\K).
\end{equation}
If, in addition, $\K=\C$,
then $\bp$ satisfies the following:
\begin{align}
&|\bp(i)|=1,
\label{bpi}\\
&|\bp(\ol{\lambda})|
=|\bp(\lambda)|
\qquad(\lambda\in\C),
\label{conjugateabsolutevalue}
\\
&\Re\bp(it)=0
\qquad(t\in\R).
\label{rebpit}
\end{align}
\end{lem}

\begin{proof}
Since $\ap$ is additive, so is $\bp$.
Moreover, by definition,
$\bp(1)=1$.

By Lemma~\ref{isolatedform}
and the fact that $|\ap(1)|=1$, 
\[
    |\bp(\lam^2)|
    =
    \frac{|\ap(\lam^2)|}{|\ap(1)|}
    =
    \frac{|\ap(\lam)|^2}{|\ap(1)|^2}
    =
    |\bp(\lam)|^2
\]
for every $\lam\in \K$.

When $\K=\C$, we have also
$|\bp(i)|=1$, since
$|\ap(1)|=1=|\ap(i)|$
by Lemma~\ref{isolatedform}.
Since $|\ap(\ol{\lam})|=|\ap(\lam)|$
by Lemma~\ref{isolatedform},
we have
$|\bp(\ol{\lam})|=|\bp(\lam)|$
for every $\lam\in \C$.

Fix $t\in \R$.
By \eqref{conjugateabsolutevalue},
$|\bp(1+it)|=|\bp(1-it)|$.
Since $\bp$ is additive and $\bp(1)=1$, we obtain 
  \begin{align*}
|\bp(1+it)|^2
& =
1+2\Re\bp(it)+|\bp(it)|^2,\\
|\bp(1-it)|^2
& =
1-2\Re\bp(it)+|\bp(it)|^2.
\end{align*}
Comparing these two equalities, we obtain
$\Re\bp(it)=0$.
\end{proof}

The next step is to show that, even in the complex case, this
normalized scalar map preserves the real line.

\begin{lem}\label{beta}
$\bp(\R)\subset \R$.
\end{lem}

\begin{proof}
If $\K=\R$, there is nothing to prove.
Assume that $\K=\C$.
Let $s\in \R$.
By \eqref{conjugateabsolutevalue},
$|\bp(s+i)|=|\bp(s-i)|$.
Since $\bp$ is additive, we have
\begin{align*}
|\bp(s+i)|^2
&=
|\bp(s)|^2+2\Re\big(\bp(s)\ol{\bp(i)}\big) + |\bp(i)|^2,\\
|\bp(s-i)|^2
&=
|\bp(s)|^2-2\Re\big(\bp(s)\ol{\bp(i)}\big) + |\bp(i)|^2.
\end{align*}
Combining these equalities gives
$\Re(\bp(s)\ol{\bp(i)})=0$.
    Since $|\bp(i)|=1$ and
$\Re\bp(i)=0$ by \eqref{bpi}
and \eqref{rebpit},
we obtain
$\bp(i)\in\{\pm i\}$.
Then $\Re\big(\bp(s)\ol{\bp(i)}\big)=0$ implies
$\Im\bp(s)=0$.
    Hence $\bp(s)\in\R$.
Therefore $\bp(\R)\subset \R$.
  \end{proof}

We now determine the normalized scalar map on the real line.
Using the square identity obtained above, we first derive
monotonicity of $\bp$ on $\R$, and then use the density of $\Q$
in $\R$.

\begin{lem}\label{lem:beta_identity}
    The map $\bp$ satisfies $\bp(s)=s$
for every $s\in \R$. 
\end{lem}

\begin{proof}
By Lemma~\ref{lem:betait},
the map $\bp$ is additive and
satisfies $\bp(1)=1$.
Hence $\bp(n)=n$
for every $n\in\N$.
    Fix $s\in \R$. 
We first prove $\bp(s^2)=\bp(s)^2$.

Since $\bp$ is additive,
we have $\bp(2ns)=2n\bp(s)$,
and thus
\[
|\bp((s+n)^2)|
=|\bp(s^2+2ns+n^2)|
=|\bp(s^2)+2n\bp(s)+n^2|
\qquad(n\in\N).
\]
By Lemma~\ref{beta},
$\bp(s),\bp(s^2)\in\R$.
We can choose $m\in\N$ such that
\[
\bp(s^2)+2m\bp(s)+m^2>0.
\]
For such $m\in\N$, we have
\[
|\bp((s+m)^2)|
=\bp(s^2)+2m\bp(s)+m^2.
\]
On the other hand,
the additivity of $\bp$,
together with
\eqref{betasquared} and
Lemma~\ref{beta}, gives
\begin{equation*}
|\bp((s+m)^2)|
=|\bp(s+m)|^2
=(\bp(s)+m)^2
=\bp(s)^2+2m\bp(s)+m^2.
\end{equation*}
    Combining these two equalities shows that
$\bp(s^2)=\bp(s)^2\geq0$
for every $s\in\R$.

    We next show that $\bp$ is monotone increasing. 
For each $s,t\in\R$ with $s\leq t$, we have
\[
\bp(t-s)
=\bp\big((\sqrt{t-s})^2\big)
=\bp(\sqrt{t-s})^2\geq0.
\]
The additivity of $\bp$ gives
$\bp(t)-\bp(s)=\bp(t-s)\geq0$,
which proves that
$\bp$ is monotone increasing.

It remains to show that $\bp(s)=s$
for every $s\in \R$.
Fix $s\in \R$. 
Since $\Q$ is dense in $\R$,
for each $k\in\N$ there exist
$q_k,r_k\in\Q$ such that
\[
s-\frac{1}{k}
\leq q_k\leq s\leq r_k
\leq s+\frac{1}{k}.
\]
The additivity of $\bp$ together with
$\bp(1)=1$ implies
$\bp(q)=q$ for every $q\in\Q$.
Since $\bp$ is additive and
monotone increasing,
we have
\[
s-\frac{1}{k}
\leq
q_k
=\bp(q_k)
\leq\bp(s)
\leq
\bp(r_k)
=r_k
\leq
s+\frac{1}{k}
\]
for every $k\in\N$.
Letting $k\to\infty$, we conclude
that $\bp(s)=s$ for every $s\in\R$.
\end{proof}

We now use this description of $\bp$ to prove that $\phi_p$,
defined as in \eqref{phiy}, is bounded
for every isolated point $p\in Y$.

\begin{lem}\label{boundedisolated}
    For every isolated point $p\in Y$,
the map $\phi_p$ is bounded. 
\end{lem}

\begin{proof}
We first prove that $|\bp(\lam)|=|\lam|$
for every $\lam\in\K$.
If $\K=\R$, this follows immediately
from Lemma~\ref{lem:beta_identity}.
Assume now that $\K=\C$.
Let $\lam=s+it\in\C$, where $s,t\in\R$.
The additivity of $\bp$ gives
\[
\bp(\lam)=\bp(s)+\bp(it).
\]
By Lemma~\ref{lem:beta_identity}
and \eqref{rebpit},
we have $\bp(s)=s$ and $\Re\bp(it)=0$.
Thus
\[
\Re\bigl(\bp(s)\ol{\bp(it)}\bigr)
=s\Re\bp(it)=0.
\]
   Applying Lemma~\ref{lem:beta_identity}
and \eqref{betasquared}, we obtain 
\[
|\bp(it)|^2
=|\bp((it)^2)|
=|\bp(-t^2)|
=|t|^2.
\]
Combining the previous equalities gives
\[
|\bp(\lam)|^2
=|\bp(s)|^2
+2\Re\bigl(\bp(s)\ol{\bp(it)}\bigr)
+|\bp(it)|^2
=|s|^2+|t|^2
=|\lambda|^2,
\]
which yields $|\bp(\lam)|=|\lam|$
for every $\lam\in\C$.

By definition, $\bp(\lam)=\ap(\lam)/\ap(1)$.
Since $|\ap(1)|=1$ by Lemma~\ref{isolatedform},
the equality $|\bp(\lam)|=|\lam|$ implies
\[
|\ap(\lam)|=|\lam|
\qquad(\lam\in\K).
\]
Therefore, by Lemma~\ref{formoff},
we have
\[
|\phi_p(f)|
=
|\HT(f)(p)|
=
|\ap\big(f(\sigma(p))\big)|
=
|f(\sigma(p))|
\leq \norm{f}
\]
for every $f\in A$.
Hence $\phi_p$ is bounded.
\end{proof}

\subsection*{Accumulation points}

We now investigate the boundedness
of $\phi_y$ when $y$ is an accumulation point of $Y$. 
Throughout this subsection, let $y\in Y$ be an accumulation point.

We first record a boundedness consequence of the separating
property.  It shows that a summable sequence of functions with
pairwise disjoint supports cannot have images with unbounded norms.

\begin{lem}\label{akbounded}
Let $\{a_k\}_{k\in\N}$ be a sequence in $A$
with pairwise disjoint supports.
If $\sum_{j=1}^\infty\norm{a_j}<\infty$,
then the sequence $\{\norm{\HT(a_k)}\}_{k\in\N}$
is bounded.
\end{lem}

\begin{proof}
Since $\sum_{j=1}^\infty\norm{a_j}<\infty$,
the series $\sum_{j=1}^\infty a_j$ converges
in $A$.
We set $a=\sum_{j=1}^{\infty}a_j$.
Hence $a\in A$.

Since the supports $\supp(a_j)$ are pairwise disjoint,
we have $a_ka_j=0$ whenever $j\neq k$. 
Using the norm convergence of the series,
we obtain, for each $k\in\N$,
\[
a_k(a-a_k)
=\sum_{j\neq k}a_ka_j=0.
\]
Since $\HT$ is additive and separating,
we obtain
$\HT(a_k)\bigl(\HT(a)-\HT(a_k)\bigr)=0$
for $k\in\N$.
Thus
\[
(\HT(a_k))^2=\HT(a_k)\HT(a)
\qquad(k\in\N).
\]
Taking the supremum norm
on both sides yields
\[
\norm{\HT(a_k)}^2
=\norm{(\HT(a_k))^2}
=\norm{\HT(a_k)\HT(a)}
\leq \norm{\HT(a_k)}\norm{\HT(a)}
\qquad(k\in\N).
\]
We have 
$\norm{\HT(a_k)}\leq\norm{\HT(a)}$
for all $k\in\N$.
Indeed, let $k\in\N$.
If $\HT(a_k)=0$, then the inequality holds.
If $\HT(a_k)\neq0$,
then the preceding inequality, dividing
by $\norm{\HT(a_k)}$, gives
$\norm{\HT(a_k)}\leq\norm{\HT(a)}$.
Thus $\{\norm{\HT(a_k)}\}_{k\in\N}$ is
bounded.
\end{proof}

We shall also need a way to separate infinitely many distinct
points of $X$ by functions with pairwise disjoint supports.
The following proposition provides such a construction.

\begin{prop}\label{propositionfunctions}
    Let $\set{x_n}_{n\in\N}$ be a sequence
of distinct points in $X$.
    Then there exist a subsequence
$\set{x_{n_k}}_{k\in \N}$ of $\set{x_n}_{n\in\N}$,
a sequence of open subsets
$\set{W_k}_{k\in\N}$ in $X$,
and a sequence $\set{h_k}_{k\in \N}$
in $A$ such that 

\begin{itemize}
\item[(i)]
$x_{n_k}\in W_k$ for every $k\in\N$,

\item[(ii)]
$h_k=1$ on $W_k$ and
$\norm{h_k}=1$ for every $k\in \N$,

\item[(iii)] $\supp(h_k)\cap \supp(h_l)
=\emptyset$ whenever $k,l\in\N$ and $k\neq l$.
\end{itemize} 
\end{prop}

\begin{proof}
Set $S=\set{x_n:n\in\N}$.
We first construct a subsequence
$\{x_{n_k}\}_{k\in\N}$ of $\{x_n\}_{n\in\N}$
and a sequence of pairwise disjoint
open subsets
$\{U_k\}_{k\in\N}$ of $X$ such that
$x_{n_k}\in U_k$ for every $k\in\N$.
Throughout this proof, closures are taken in $\X$.

\textbf{Step 1.}
Suppose that $S$ has an accumulation
point $z_0$ in $X$.
Set $V_0=X$ and $n_0=0$.
We inductively construct a subsequence
$\set{x_{n_k}}_{k\in\N}$,
a sequence $\set{V_k}_{k\in\N}$
of open neighborhoods of $z_0$ in $\X$,
and a sequence $\set{U_k}_{k\in\N}$ of
open subsets of $\X$ contained in
$X$ such that
\[
x_{n_k}\in U_k,\qquad
\overline{U_k}\subset V_{k-1}\setminus\{z_0\},\qquad
V_k=V_{k-1}\setminus\overline{U_k}
\]
for every $k\in\N$.
    Since $z_0$ is an accumulation point of $S$, the set
    $V_0\cap(S\setminus\set{z_0})$ is infinite.
    Choose
    \[
        x_{n_1}\in V_0\cap(S\setminus\set{z_0}).
    \]
Since $\X$ is compact Hausdorff,
there exists an
open neighborhood $U_1$ of $x_{n_1}$
in $\X$ such that
\[
\ol{U_1}\subset V_0\setminus\{z_0\}.
\]
In particular, $U_1\subset X$.
Set $V_1=V_0\setminus\overline{U_1}$.
Then $V_1$ is an open neighborhood of $z_0$.

Assume that we have chosen
$x_{n_j},U_j,V_j$ so that
\[
x_{n_j}\in U_j,\quad
n_j>n_{j-1},\qquad
\ol{U_j}\subset V_{j-1}\setminus\{z_0\},\qquad
V_j=V_{j-1}\setminus\ol{U_j}
\]
for $j=1,\dots,k$.
Since $z_0$ is an accumulation point
of $S$, $V_k$ contains
infinitely many points of
$S\setminus\{z_0\}$.
Choose
\[
x_{n_{k+1}}\in V_k\cap(S\setminus\set{z_0})
\]
with $n_{k+1}>n_k$.
Since $\X$ is compact Hausdorff,
there exists an
open neighborhood $U_{k+1}$ of $x_{n_{k+1}}$
such that
\[
\ol{U_{k+1}}\subset V_k\setminus\{z_0\}.
\]
In particular, $U_{k+1}\subset X$.
Set $V_{k+1}=V_k\setminus\overline{U_{k+1}}$.
This completes the inductive construction.

We show that the sets $U_k$ are pairwise disjoint.
Let $l,m\in\N$ with $l< m$.
Since $\ol{U_k}\subset V_{k-1}$ for every $k\in\N$,
we obtain $U_m\subset V_l$ by induction.
On the other hand,
$V_l=V_{l-1}\setminus\ol{U_l}$ shows
$U_l\cap V_l=\emptyset$.
This proves that $U_l\cap U_m=\emptyset$.
Therefore $\{U_k\}_{k\in\N}$ is pairwise disjoint.

\textbf{Step 2.}
Suppose that $S$ has no accumulation point in $X$.
We inductively construct a sequence
$\set{V_k}_{k\in\N}$
of open neighborhoods of $x_k$ in $\X$,
and a sequence $\set{U_k}_{k\in\N}$
of pairwise disjoint
open subsets of $\X$ contained in
$X$ such that
\[
x_k\in U_k,\qquad
\overline{U_k}\subset V_k,\qquad
V_k\cap S=\{x_k\},\qquad
V_k\subset X\setminus\bigcup_{j=1}^{k-1}\overline{U_j}
\]
for every $k\in\N$,
where we assume $\bigcup_{j=1}^0\ol{U_j}=\emptyset$.
Since $x_1$ is not an accumulation point of $S$ in $X$,
there exists an open neighborhood $V_1$ of $x_1$
in $\X$ such that
\[
V_1\subset X,\qquad
V_1\cap S=\{x_1\}.
\]
Since $\X$ is compact Hausdorff, there exists
an open neighborhood
$U_1$ of $x_1$ such that 
\[
\ol{U_1}\subset V_1.
\]

Suppose that we have chosen
open neighborhoods $V_n$ of $x_n$ and
pairwise disjoint
open neighborhoods $U_n$
of $x_n$ such that
\[
x_n\in U_n,\qquad
\ol{U_n}\subset V_n,\qquad
V_n\cap S=\{x_n\},
\qquad
V_n\subset
X\setminus\bigcup_{j=1}^{n-1}\overline{U_j}
\]
for $n=1,\dots,k$.
Since $\ol{U_n}\subset V_n$ and
$V_n\cap S=\{x_n\}$ for $n=1,\ldots,k$,
we have
$x_{k+1}\notin\bigcup_{n=1}^k \overline{U_n}$.
Since $x_{k+1}$ is not an accumulation point
of $S$ in $X$, there exists
an open neighborhood $V_{k+1}$ of $x_{k+1}$
in $\X$ such that
\[
V_{k+1}\cap S=\{x_{k+1}\},\qquad
V_{k+1}\subset
X\setminus\bigcup_{j=1}^k\overline{U_j}.
\]
Since $\X$ is compact Hausdorff,
there exists an
open neighborhood $U_{k+1}$
of $x_{k+1}$ such that
\[
\ol{U_{k+1}}\subset V_{k+1}.
\]
Then $\ol{U_{k+1}}\subset
X\setminus\bigcup_{n=1}^k\overline{U_n}$.
Thus the sets $U_1,\ldots,U_{k+1}$ are
pairwise disjoint.
This completes the induction.
In this case, we set
$n_k=k$ for every $k\in\N$.

\textbf{Step 3.}
Since $\X$ is a
compact Hausdorff space,
we can choose an
open subset $W_k$ of $\X$
such that 
    \begin{equation*}
        x_{n_k}\in W_k\subset \ol{W_k}\subset U_k
    \end{equation*}
    for each $k\in \N$.
    By Urysohn's lemma, for each $k\in \N$,
there exists $h_k\in A$ such that 
\[
h_k=1\qbox{on} \ol{W_k},\qquad
\norm{h_k}=1,\qquad
\supp(h_k)\subset U_k.
\]
    Since the sets $U_k$ are pairwise disjoint,
the sets $\supp(h_k)$ are pairwise disjoint. 
    Thus $\set{x_{n_k}}_{k\in\N}$,
$\set{W_k}_{k\in\N}$, and $\set{h_k}_{k\in\N}$
have the desired properties.
\end{proof}

We now show that, if $\phi_y$ were unbounded, then one could
construct a summable sequence of functions with pairwise disjoint
supports whose images have arbitrarily large values.
This will contradict Lemma~\ref{akbounded}.

\begin{lem}\label{unboundedproperty}
If $\phi_y$ is unbounded, then there exist
sequences $\{a_k\}_{k\in\N}$ in $A$
and $\{z_k\}_{k\in\N}$ in $Y$ such that
\begin{itemize}
\item[(i)]
$\norm{a_k}\le 2^{-k}$ for every $k\in\N$,
    
\item[(ii)]
$|\HT(a_k)(z_k)|>2^{k-1}$
for every $k\in\N$,
    
\item[(iii)]
$\supp(a_k)\cap\supp(a_l)=\emptyset$
whenever $k,l\in\N$ and $k\neq l$.
\end{itemize}
\end{lem}

\begin{proof}
We recall that $\phi_y(f)=\HT(f)(y)$ for $f\in A$
by \eqref{phiy}.
By assumption, $\phi_y$ is unbounded.
Then, for each $n\in\N$, there exists
$f_n\in A\setminus\{0\}$ such that
\[
|\HT(f_n)(y)|>4^n\norm{f_n}.
\]
Let $q_n\in\Q$ be such that
\[
4^{-n}|\HT(f_n)(y)|>q_n\geq\norm{f_n}>0.
\]
In particular,
\[
\frac{1}{q_n}\leq\frac{1}{\norm{f_n}},\qquad
\frac{|\HT(f_n)(y)|}{q_n}>4^n
\qquad(n\in\N).
\]
Set $g_n=f_n/(2^nq_n)$
for each $n\in\N$.
Then $g_n\in A$ and
\[
\norm{g_n}
=\frac{\norm{f_n}}{2^nq_n}
\leq\frac{\norm{f_n}}{2^n\norm{f_n}}
=2^{-n}
\qquad(n\in\N).
\]
Using the $\Q$-linearity of $\HT$ and the choice of $q_n$,
we have
\[
|\HT(g_n)(y)|
=\left|\HT\left(\frac{f_n}{2^nq_n}\right)(y)\right|
=\frac{1}{2^n}\,\frac{|\HT(f_n)(y)|}{q_n}
>\frac{4^n}{2^n}=2^n
\qquad(n\in\N).
\]
Consequently, we obtain
\begin{equation}\label{sequenceg_n}
\|g_n\|\le 2^{-n},\qquad
|\HT(g_n)(y)|>2^n
\qquad(n\in\N).
\end{equation}

\textbf{Step 1.}
We construct a sequence
$\set{y_n}_{n\in \N}$ in $Y$
such that $|\HT(g_n)(y_n)|>2^{n-1}$
for every $n\in \N$.

For each $n\in\N$, define an open subset
$O_n$ by
\[
O_n=\{z\in Y:|\HT(g_n)(z)|>2^{n-1}\}.
\]
Since $|\HT(g_n)(y)|>2^n$
by \eqref{sequenceg_n},
the set $O_n$ is an open neighborhood of $y$.
    Moreover, since $y$ is an accumulation point of $Y$,
$O_n$ is infinite.

We choose $y_n\in\sigma^{-1}(X)\cap O_n$ inductively.
By Lemma~\ref{lem:denseopen},
$\sigma^{-1}(X)$ is a dense open
subset of $Y$.
Since $O_1$ is infinite,
$O_1\setminus\{y\}$ is a non-empty
open subset of $Y$.
Thus
$\sigma^{-1}(X)\cap(O_1\setminus\{y\})
\neq\emptyset$.
We can choose
$y_1\in\sigma^{-1}(X)\cap(O_1\setminus\{y\})$.
Suppose that $y_1,\dots,y_n$ have
already been chosen so that
\[
y_k\in
\sigma^{-1}(X)\cap(O_k\setminus\{y\})
\qquad(k=1,\dots,n),
\]
and $y_k\neq y_l$ whenever $k\neq l$.
Set $O_{n+1}'=O_{n+1}\setminus\{y_1,\dots,y_n\}$.
Then $O_{n+1}'$ is an open neighborhood
of $y$. Since $O_{n+1}$ is infinite,
so is $O_{n+1}'$.
Hence $O_{n+1}'\setminus\{y\}$ is a non-empty
open subset of $Y$.
Since $\sigma^{-1}(X)$ is dense,
we have $\sigma^{-1}(X)\cap(O_{n+1}'\setminus\{y\})\neq\emptyset$.
Therefore we can choose
$y_{n+1}\in\sigma^{-1}(X)\cap(O_{n+1}'\setminus\{y\})$.
Thus
\[
y_{n+1}\in\sigma^{-1}(X)\cap(O_{n+1}\setminus\{y,y_1,\dots,y_n\}).
\]
This gives a sequence $\{y_n\}_{n\in\N}$ of
pairwise distinct points such that
$y_n\in\sigma^{-1}(X)\cap O_n$
for every $n\in\N$.

By the defining property of $O_n$, we have
\begin{equation}\label{eq:yn}
|\HT(g_n)(y_n)|>2^{n-1}
\qquad(n\in\N).
\end{equation}
Since $y_n\in\sigma^{-1}(X)$,
we have $\sigma(y_n)\in X$
for every $n\in\N$.
Moreover, since the points
$y_n$ are pairwise distinct
and $\sigma$ is injective on
$\sigma^{-1}(X)$ by Lemma~\ref{bijection},
the points $\sigma(y_n)$ are pairwise distinct.

\textbf{Step 2.}
We construct a sequence $\set{a_k}_{k\in \N}$
in $A$ and a subsequence $\set{y_{n_k}}_{k\in\N}$
of $\{y_n\}_{n\in\N}$ with the required properties
(i) through (iii).

    Set $x_n=\sigma(y_n)$ for each $n\in \N$.
    By Proposition~\ref{propositionfunctions},
we obtain a subsequence $\set{x_{n_k}}_{k\in \N}$
of $\set{x_n}_{n\in\N}$,
a sequence of open subsets
$\set{W_k}_{k\in\N}$ of $X$,
and a sequence $\set{h_k}_{k\in \N}$
in $A$ such that
\[
x_{n_k}\in W_k,\qquad
h_k=1\qbox{on}W_k,\qquad
\norm{h_k}= 1
\]
for every $k\in \N$,
and the sets $\supp(h_k)$ are pairwise disjoint.

Set
\[
a_k=g_{n_k}h_k
\qquad(k\in \N).
\]
Then $a_k\in A$.
Since $\{x_{n_k}\}_{k\in\N}$ is a subsequence
of $\{x_n\}_{n\in\N}$, we obtain
$k\leq n_k$ for every $k\in\N$.
By \eqref{sequenceg_n}, we have 
\[
\norm{a_k}
\leq \norm{g_{n_k}}\,\norm{h_k}
=\norm{g_{n_k}}
\leq 2^{-n_k}\leq 2^{-k}
\qquad(k\in\N).
\]
Thus the sequence $\{a_k\}_{k\in\N}$
satisfies (i).

Since $\supp(a_k)\subset\supp(h_k)$
for every $k\in\N$,
and the sets $\supp(h_k)$ are pairwise disjoint,
we have
$\supp(a_k)\cap\supp(a_l)=\emptyset$
for every $k,l\in\N$ with $k\neq l$.
This proves (iii).

Let $k\in\N$.
Since $h_k=1$ on $W_k$,
we have 
\[
(g_{n_k}-a_k)(x)
=g_{n_k}(x)(1-h_k(x))=0
\qquad(x\in W_k).
\]
Since $x_{n_k}\in W_k$,
we obtain
$\sigma(y_{n_k})=x_{n_k}
\notin\supp(g_{n_k}-a_k)$. 
By the contrapositive of
\eqref{supp}, we have
$\HT(g_{n_k}-a_k)(y_{n_k})=0$.
The additivity of $\HT$ yields
$\HT(a_k)(y_{n_k})=\HT(g_{n_k})(y_{n_k})$.
By \eqref{eq:yn}, we have
$|\HT(g_{n_k})(y_{n_k})|>2^{n_k-1}$.
Hence
\[
|\HT(a_k)(y_{n_k})|
=|\HT(g_{n_k})(y_{n_k})|
>2^{n_k-1}
\geq2^{k-1}.
\]
Setting $z_k=y_{n_k}$,
the sequence $\set{z_k}_{k\in\N}$
satisfies (ii).
\end{proof}

We can now rule out the possibility that $\phi_y$ is unbounded
at an accumulation point.

\begin{lem}\label{boundedaccumulation}
If $y$ is an accumulation point of $Y$,
then $\phi_y$ is bounded.
\end{lem}

\begin{proof}
Suppose that $\phi_y$ is unbounded.
By Lemma~\ref{unboundedproperty},
there exist sequences $\{a_k\}_{k\in\N}$
in $A$ and $\{z_k\}_{k\in\N}$ in $Y$
such that
\[
\norm{a_k}\le 2^{-k},
\qquad
|\HT(a_k)(z_k)|>2^{k-1}
\]
for every $k\in\N$, and the supports
of the functions $a_k$ are
pairwise disjoint.
In particular,
$\sum_{k=1}^{\infty}\norm{a_k}<\infty$
and
$\norm{\HT(a_k)}
\ge |\HT(a_k)(z_k)|
>2^{k-1}$
for every $k\in\N$.
This contradicts Lemma~\ref{akbounded}.
\end{proof}

\section{Proof of the main results}
\label{sect:proof}

\begin{proof}
[\textbf{Proof of Theorems~\ref{thm:complex} and \ref{thm:real}}]
We first show that $\phi_y\colon A\to\K$,
defined as in \eqref{phiy}, is a bounded
real-linear map for every $y\in Y$.
For each $y\in Y$, the boundedness of $\phi_y$
follows from Lemmas~\ref{boundedisolated}
and \ref{boundedaccumulation},
since every point of $Y$
is either isolated or an accumulation point.
It remains to prove that $\phi_y$ is real-linear.
Since $\HT$ is additive, so is $\phi_y$
by \eqref{phiy}.
Thus $\phi_y$ is $\Q$-linear.
For each $s\in\R$, there exists a sequence
$\{q_n\}_{n\in\N}$ in $\Q$ such that
$q_n\to s$ as $n\to\infty$.
Since $\phi_y$ is bounded, there exists $C_y>0$ such that
\[
|\phi_y(g)|\leq C_y\norm{g}
\qquad(g\in A).
\]

Since $\phi_y$ is
$\Q$-linear, for every $f\in A$ we have
\begin{align*}
|\phi_y(sf)-s\phi_y(f)|
&\leq
|\phi_y(sf)-\phi_y(q_nf)|
+|q_n\phi_y(f)-s\phi_y(f)|\\
&\leq
C_y\norm{(s-q_n)f}
+|q_n-s|\,C_y\norm{f}\\
&\to0
\qquad(n\to\infty).
\end{align*}
This proves that $\phi_y$ is real-linear.

For each $f\in A$, we have
\[
\sup_{y\in Y}|\phi_y(f)|
=\sup_{y\in Y}|\HT(f)(y)|
=\norm{\HT(f)}<\infty.
\]
Thus the family $\{\phi_y\}_{y\in Y}$ is a pointwise
bounded family of bounded
real-linear maps from $A$, regarded as a real
Banach space, into
the real Banach space $\K$.
Applying the Banach--Steinhaus theorem,
there exists $M>0$ such that
\[
|\phi_y(f)|\le M\norm{f}
\qquad(f\in A,\ y\in Y).
\]
    Taking the supremum over $y\in Y$, we obtain 
\begin{equation*}
\norm{\HT(f)}
=\sup_{y\in Y}|\HT(f)(y)|
=\sup_{y\in Y}|\phi_y(f)|
\leq M\norm{f}
\qquad(f\in A),
\end{equation*}
which implies that $\HT$ is bounded.
Since $\HT$ is bounded and additive, the same
approximation argument used above for $\phi_y$
shows that $\HT$ is real-linear.

    We next show that $\HT$ is norm-preserving.
Since $\HT$ is additive, we have
$\HT(0)=0$.
Fix an arbitrary $f\in A\setminus\{0\}$.
We prove $\norm{\HT(f)}=\norm{f}$.
By induction, using the product norm identity, we obtain
\[
\norm{\HT(f^{2^n})}=\norm{\HT(f)}^{2^n}
\qquad(n\in\N).
\]
    Hence
    \begin{equation*}
        \norm{\HT(f)}^{2^n}
        =\norm{\HT(f^{2^n})}
        \leq \norm{\HT}\,\norm{f^{2^n}}
        =\norm{\HT}\,\norm{f}^{2^n}.
    \end{equation*}
    Taking the $2^n$th root of both sides yields
$\norm{\HT(f)}\leq \norm{\HT}^{1/2^n}\norm{f}$.
    Letting $n\to \infty$, we obtain
    $\norm{\HT(f)}\leq \norm{f}$.

Since $\HT$ is a bounded real-linear bijection
between Banach spaces,
the open mapping theorem implies that $\HT^{-1}$ is bounded.
    Hence there exists $C>0$ such that
$\norm{f}\leq C\norm{\HT(f)}$.
By the product norm identity for $\HT$,
we have
\[
\norm{f}^{2^n}
=\norm{f^{2^n}}
\leq C\norm{\HT(f^{2^n})}
=C\norm{\HT(f)}^{2^n}
\qquad(n\in\N).
\]
    Taking the $2^n$th root of both sides, we obtain
$\norm{f}\leq C^{1/2^n}\norm{\HT(f)}$. 
    Letting $n\to \infty$, we have $\norm{f}\leq\norm{\HT(f)}$. 
Combining the two inequalities gives
$\norm{\HT(f)}=\norm{f}$.
Therefore $\HT\colon A\to B$ is
a surjective real-linear isometry.

For each $f\in A$, we have $f|_X\in C_0(X,\K)$
and
\[
\HT(f)(y)=
\begin{cases}
T(f|_X)(y),&y\in Y,\\
0,&y=\yinf,
\end{cases}
\]
by \eqref{HT}.
Through the above isometric identifications of
$C_0(X,\K)$ with $A$ and $C_0(Y,\K)$ with $B$,
it follows that
$T\colon C_0(X,\K)\to C_0(Y,\K)$ is a surjective
real-linear isometry.

When $\K=\C$, it follows from
\cite[Theorem~1.1]{miurareallinear} that
there exist a continuous function
$w\colon Y\to\T$,
a homeomorphism $\varphi\colon Y\to X$,
and a closed and open subset
$Y_0\subset Y$ such that
\[
T(f)(y)=
\begin{cases}
w(y)f(\varphi(y)),&y\in Y_0,\\
w(y)\ol{f(\varphi(y))},&y\in Y\setminus Y_0,
\end{cases}
\]
for every $f\in C_0(X,\C)$ and $y\in Y$.

Conversely, suppose that
$T\colon C_0(X,\C)\to C_0(Y,\C)$
is a map of the above form.
A direct verification shows that
$T$ is a ring isomorphism in norm
and satisfies $\norm{T(\ol{f})}=\norm{T(f)}$
for every $f\in C_0(X,\C)$.

When $\K=\R$, the Banach--Stone theorem
shows that there exist a continuous function
$w\colon Y\to\{\pm1\}$ and a homeomorphism
$\varphi\colon Y\to X$ such that
\[
T(f)(y)=w(y)f(\varphi(y))
\]
for every $f\in C_0(X,\R)$ and $y\in Y$.

Conversely, suppose that
$T\colon C_0(X,\R)\to C_0(Y,\R)$
is a map of the above form.
A direct verification shows that
$T$ is a ring isomorphism in norm.
\end{proof}

\begin{rem*}
The homeomorphism $\varphi$ appearing
in the final representation
coincides with the map $\sigma$ constructed
in Section~\ref{sect:additive}.
Indeed, let $y\in Y$ and $f\in\F_y$.
Put $g=f|_X$.
Then $g\in C_0(X,\K)$ and
\[
T(g)(y)=\HT(f)(y)\neq0.
\]
By the final representation of $T$,
in both the real and complex cases,
this implies that
$g(\varphi(y))\neq0$.
Since $g=f|_X$ and $\varphi(y)\in X$, we have
$f(\varphi(y))\neq0$.
Thus
\[
\varphi(y)\in\coz(f)\subset\supp(f)
\qquad(f\in\F_y).
\]
Since $f\in\F_y$ was arbitrary, it follows that
$\varphi(y)\in\bigcap_{f\in\F_y}\supp(f)=\{\sigma(y)\}$.
Therefore $\varphi(y)=\sigma(y)$ for every $y\in Y$.
\end{rem*}

\subsection*{Acknowledgments}
The second author was supported by JST SPRING,
Grant Number JPMJSP2121.
The third author was supported by JSPS KAKENHI
Grant Number JP25K07028.

\end{document}